\documentclass[12pt]{article}
\usepackage{amsmath,amssymb,amsthm}
\numberwithin{equation}{section}
\usepackage{units}
\usepackage{color}
\usepackage[T1]{fontenc}
\usepackage[utf8]{inputenc}
\usepackage{authblk}
\usepackage{bm}
\usepackage{graphicx}

\usepackage{datetime}

\usepackage[colorlinks=true,
linkcolor=webgreen,
filecolor=webbrown,
citecolor=webgreen]{hyperref}

\definecolor{webgreen}{rgb}{0,.5,0}
\definecolor{webbrown}{rgb}{.6,0,0}

\usepackage[toc,page]{appendix}

\newcommand{\red}[1]{{\color{red}#1}}

\newcommand{\Z}{{\mathbb Z}}

\newtheorem{thm}{Theorem}
\newtheorem{theorem}[thm]{Theorem}
\newtheorem{lemma}{Lemma}

\newtheorem{corollary}[thm]{Corollary}

\title{Squares of Fibonacci-Like Numbers}
\author[1]{Kunle Adegoke\thanks{Corresponding author: \href{mailto:adegoke00@gmail.com}{\tt adegoke00@gmail.com}}}
\affil{Department of Physics and Engineering Physics, \mbox{Obafemi Awolowo University}, 220005 Ile-Ife, Nigeria}
\author{Tokunbo Omiyinka}
\affil{Department of Physics, \mbox{University of Alberta}, Edmonton, Alberta, T6G 2E1, Canada}

\begin{document}

\date{}

\maketitle

\begin{abstract}
\noindent We derive a general recurrence relation for squares of Fibonacci-like numbers. Various properties are developed, including double binomial summation identites. 
\end{abstract}
\section{Introduction}
The Fibonacci numbers, $F_j$, and the Lucas numbers, $L_j$, $j\in\Z$, are defined by:
\begin{equation}
F_0  = 0,\;F_1  = 1,\;F_j  = F_{j - 1}  + F_{j - 2}\; (j \ge 2),\quad F_{ - j}  = ( - 1)^{j-1} F_j
\end{equation}
and
\begin{equation}
L_0  = 2,\;L_1  = 1,\;L_j  = L_{j - 1}  + L_{j - 2}\; (j \ge 2) ,\quad L_{ - j}  = ( - 1)^j L_j\,.
\end{equation}
Both $(F_j)_{j\in\Z}$ and $(L_j)_{j\in\Z}$ are examples of a Fibonacci-like sequence. We define a Fibonacci-like sequence, $(G_j)_{j\in\Z}$, as one having the same recurrence relation as the Fibonacci sequence, but with arbitrary initial terms. Thus, given arbitrary integers $G_0$ and $G_1$, not both zero, we define
\begin{equation}
G_j=G_{j-1}+G_{j-2}\; (j \ge 2)\,;
\end{equation}
and also extend the definition to negative subscripts by writing the recurrence relation as
\begin{equation}\label{eq.ozw7j10}
\quad G_{-j}=G_{-j+2}-G_{-j+1}\,.
\end{equation}
We have \cite[equation (1.5)]{adegoke18f}
\begin{equation}
G_{-j}=(-1)^j(G_0L_j-G_j)\,.
\end{equation}
The identity (see Brousseau \cite[equation (2)]{brousseau})
\begin{equation}
F_{j - 1}^2  + F_{j + 2}^2  = 2F_j^2  + 2F_{j + 1}^2\,,
\end{equation}
or, more generally,
\begin{equation}\label{eq.vcah8wn}
G_{j - 1}^2  + G_{j + 2}^2  = 2G_j^2  + 2G_{j + 1}^2\,,
\end{equation}
is well known.

\medskip

Less familiar are identities such as
\begin{equation}\label{eq.vnkmyzu}
G_{j+2}^2+2G_{j-2}^2 = 3G_{j-1}^2+6G_{j}^2\,,
\end{equation}
\begin{equation}\label{eq.xvk6xej}
3G_{j+3}^2+G_{j-3}^2 = 16G_{j+1}^2+12G_{j}^2
\end{equation}
and
\begin{equation}\label{eq.aex38xp}
F_{k}F_{k+1}G_{j+1}^2-F_kF_{k-1}G_{j-1}^2 = G_{j+k}^2-F_{k+1}F_{k-1}G_{j}^2\,.
\end{equation}
Our aim in writing this paper is to derive the identity
\[
\begin{split}
F_s F_m F_{m - s} G_{j + k}^2  &= F_{m - s} F_{m - k} F_{s - k} G_j^2  + ( - 1)^{s + k} F_k F_m F_{m - k} G_{j + s}^2\\
&\qquad\qquad - ( - 1)^{s + k} F_k F_s F_{s - k} G_{j + m}^2\,,
\end{split}
\]
of which \eqref{eq.vcah8wn}, \eqref{eq.vnkmyzu}, \eqref{eq.xvk6xej} and \eqref{eq.aex38xp} are particular cases, being evaluations at certain $m$, $k$, $j$ and $s$ choices.

\medskip

Closed formulas are known for $\sum_{j=0}^n{x^jG_{j}^2}$ and $\sum_{j=0}^n{G_{j+k}G_{j-k}}$. We will extend these results by providing evaluations for $\sum_{j=0}^n{x^jG_{j+k}^2}$ and $\sum_{j=0}^n{x^jG_{j+k}G_{j+s}}$ for integers $n$, $s$ and $k$ and arbitrary $x$.

\medskip

Finally, we will derive double binomial identities involving the squares of Fibonacci-like numbers.
\section{Main identity}
\begin{theorem}\label{theorem.o1g1xkr}
If $j$, $k$, $m$ and $s$ are integers, then
\[
\begin{split}
F_s F_m F_{m - s} G_{j + k}^2  &= F_{m - s} F_{m - k} F_{s - k} G_j^2  + ( - 1)^{s + k} F_k F_m F_{m - k} G_{j + s}^2\\
&\qquad\qquad - ( - 1)^{s + k} F_k F_s F_{s - k} G_{j + m}^2\,.
\end{split}
\]
\end{theorem}
\begin{proof}
Setting $m=0$ in the identity (see Adegoke \cite[Theorem 1]{adegoke18f})
\begin{equation}\label{eq.a5x1nph}
F_{s - k} G_{j + m}  = F_{m - k} G_{j + s}  + ( - 1)^{s + k + 1} F_{m - s} G_{j + k}
\end{equation}
gives
\begin{equation}
( - 1)^k F_{s - k} G_j  = F_s G_{j + k}  - F_k G_{j + s}\,,
\end{equation}
from which, by squaring and re-arranging, we get
\begin{equation}\label{eq.cgd7f12}
2F_s F_k G_{j + k} G_{j + s}  = F_s^2 G_{j + k}^2  + F_k^2 G_{j + s}^2  - F_{s - k}^2 G_j^2\,.
\end{equation}
The statement of the theorem then follows by squaring identity \eqref{eq.a5x1nph} and using \eqref{eq.cgd7f12} to eliminate the cross-term $G_{j+k}G_{j+s}$ from the right hand side, while making use also of the multiplication formula
\begin{equation}
5F_mF_n=L_{m+n}-(-1)^nL_{m-n}\,.
\end{equation}
\end{proof}
\section{Partial sums and generating function}
\begin{lemma}[{\cite[Lemma 2]{adegoke18c}}Partial sum of a $(r+1)$-term sequence]\label{lemma.qm8k37h}
Let $(X_j)$ be any arbitrary sequence, where $X_j$, $j\in\Z$, satisfies a $(r+1)$-term recurrence relation $X_j=f_1X_{j-c_1}+f_2X_{j-c_2}+\cdots+f_rX_{j-c_r}=\sum_{m=1}^r f_mX_{j-c_m}$, where $f_1$, $f_2$, $\ldots$, $f_r$ are arbitrary non-vanishing complex functions, not dependent on $j$, and $c_1$, $c_2$, $\ldots$, $c_r$ are fixed integers. Then, the following summation identity holds for arbitrary $x$ and non-negative integer $n$ :
\[
\sum_{j = 0}^n {x^j X_j }  = \frac{{\sum_{m = 1}^r {\left\{ {x^{c_m } f_m \left( {\sum_{j = 1}^{c_m } {x^{ - j} X_{ - j} }  - \sum_{j = n - c_m  + 1}^n {x^j X_j } } \right)} \right\}} }}{{1 - \sum_{m = 1}^r {x^{c_m } f_m } }}\,.
\]

\end{lemma}
We note that a special case of Lemma~\ref{lemma.qm8k37h} was proved by Zeitlin~\cite{zeitlin64}.
\begin{theorem}\label{theorem.mofbl42}
The following identity holds for arbitrary $x$ and integers $k$ and $n$:
\[
\begin{split}
x\sum\limits_{j = 0}^n {x^j G_{j + k}^2 }  &= (xF_{k + 1} F_{k - 1}  + F_{k + 1} F_k  - x^2 F_{k - 1} F_k )S_{G,n} (x)\\
&\qquad+ xF_{k - 1} F_k (x^{n + 1} G_n^2  - G_0^2  - G_1^2  + 2G_0 G_1 )\\
&\qquad+ F_{k + 1} F_k (x^{n + 1} G_{n + 1}^2  - G_0^2 )\,,
\end{split}
\]
where
\begin{equation}\label{eq.araqz6y}
\begin{split}
S_{G,n}(x)=\sum_{j = 0}^n {x^j G_j^2 }  &=  - \frac{{(2x^2  + 2x - 1)G_0^2 }}{{x^3  - 2x^2  - 2x + 1}} - \frac{{(2x^2  - x)G_1^2 }}{{x^3  - 2x^2  - 2x + 1}}\\
&\qquad + \frac{{x^2 G_2^2 }}{{x^3  - 2x^2  - 2x + 1}} + \frac{{(2x^2  + 2x - 1)x^{n + 1} G_{n + 1}^2 }}{{x^3  - 2x^2  - 2x + 1}}\\
&\qquad + \frac{{(2x - 1)x^{n + 2} G_{n + 2}^2 }}{{x^3  - 2x^2  - 2x + 1}} - \frac{{x^{n + 3} G_{n + 3}^2 }}{{x^3  - 2x^2  - 2x + 1}}\,.
\end{split}
\end{equation}
\end{theorem}

\begin{proof}
First, the identity \eqref{eq.araqz6y}, derived in Adegoke \cite[equation (3.1)]{adegoke18a}, also follows from \eqref{eq.vcah8wn} and Lemma \ref{lemma.qm8k37h} with $X_j=G_j^2$.

\medskip

Note that
\begin{equation}\label{eq.nkzwbt5}
S_{F,n} (x) = \sum\limits_{j = 0}^n {x^j F_j^2 = } \frac{{x(1 - x) - (1 - 2x - 2x^2 )x^{n + 1} F_{n + 1}^2  - (1 - 2x)x^{n + 2} F_{n + 2}^2  - x^{n + 3} F_{n + 3}^2 }}{{(1 - 2x - 2x^2  + x^3 )}}\,,
\end{equation}
\begin{equation}\label{eq.xei24hh}
S_{G,n}(1)=\sum_{j=0}^n{G_j^2}=G_nG_{n+1}-G_0G_1+G_0^2\,,
\end{equation}
\begin{equation}
S_{F,n} (1) = \sum\limits_{j = 0}^n {F_j x^j }  = \frac{1}{2}F_{n + 3}^2  - \frac{1}{2}F_{n + 2}^2  - \frac{3}{2}F_{n + 1}^2  = F_n F_{n + 1}\,.
\end{equation}
Setting $s=1$ and $m=-1$ in the identity of Theorem \ref{theorem.o1g1xkr} and re-arranging, we have
\begin{equation}\label{eq.iwr94yf}
G_{j + k}^2  = F_{k + 1} F_{k - 1} G_j^2  + F_{k + 1} F_k G_{j + 1}^2  - F_{k - 1} F_k G_{j - 1}^2\,,
\end{equation}
which allows us to write
\begin{equation}\label{eq.bp2inyu}
\sum\limits_{j = 0}^n {x^j G_{j + k}^2 }  = F_{k + 1} F_{k - 1} \sum\limits_{j = 0}^n {x^j G_j^2 }  + F_{k + 1} F_k \sum\limits_{j = 0}^n {x^j G_{j + 1}^2 }  - F_{k - 1} F_k \sum\limits_{j = 0}^n {x^j G_{j - 1}^2 }\,.
\end{equation}
Now,
\begin{equation}\label{eq.vgd6iyo}
\sum\limits_{j = 0}^n {x^j G_{j + 1}^2 }  = \sum\limits_{j = 1}^{n + 1} {x^{j - 1} G_j^2 }  = \frac{1}{x}\left( {\sum\limits_{j = 0}^n {x^j G_j^2 }  - G_0^2  + x^{n + 1} G_{n + 1}^2 } \right)
\end{equation}
and
\begin{equation}\label{eq.ucegrs5}
\sum\limits_{j = 0}^n {x^j G_{j - 1}^2 }  = x\sum\limits_{j = -1}^{n - 1} {x^j G_j^2 }  = x\left( {\sum\limits_{j = 0}^n {x^j G_j^2 }  + \frac{1}{x}G_{ - 1}^2  - x^n G_n^2 } \right)\,.
\end{equation}
Using \eqref{eq.araqz6y}, \eqref{eq.vgd6iyo} and \eqref{eq.ucegrs5} in \eqref{eq.bp2inyu} produces the identity of Theorem \ref{theorem.mofbl42}.
\end{proof}
Observe that setting $x=-1$ in \eqref{eq.nkzwbt5} makes the right hand side to be an indeterminate form. Application of L'Hospital's rule however provides the evaluation of $S_{F,n} ( - 1)$. Thus, we have
\begin{equation}
\begin{split}
5( - 1)^{n - 1} S_{F,n} ( - 1) &= 5( - 1)^{n - 1} \sum\limits_{j = 0}^n {( - 1)^j F_j^2 }\\  
&= (n + 3)F_{n + 3}^2  - (3n + 8)F_{n + 2}^2  + (n - 1)F_{n + 1}^2  + ( - 1)^{n - 1} 3\,.
\end{split}
\end{equation}
Setting $x=1$ in the identity of Theorem \ref{theorem.mofbl42}, we have
\begin{equation}\label{eq.pl8q4s9}
\begin{split}
\sum\limits_{j = 0}^n {G_{j + k}^2 }  &=( F_{k+1}F_{k-1}+F_{k}^2 )(G_n G_{n + 1}  - G_0 G_1  + G_0^2 )\\
&\qquad + F_{k - 1} F_k (G_n  - G_1  + G_0 )(G_n  + G_1  - G_0 )\\
&\quad\qquad + F_{k + 1} F_k (G_{n + 1}  - G_0 )(G_{n + 1}  + G_0 )\,.
\end{split}
\end{equation}
\begin{theorem}[Generating function of $F_{j}^2$]
\begin{equation}\label{eq.mc19w1x}
\sum\limits_{j = 0}^\infty {x^j F_j^2 = } \frac{{x(1 - x) }}{{1 - 2x - 2x^2  + x^3 }}\,.
\end{equation}
\end{theorem}
\begin{proof}
Identity \eqref{eq.nkzwbt5} as $n$ approaches infinity; with $x^nF_n^2\to 0$ as $n$ approaches infinity.
\end{proof}
Next, we provide an alternative evaluation of $\sum_{j = 0}^n {x^j G_{j + k}^2 }$, not requiring the initial values $G_0$ and $G_1$ of the sequence $(G_r)_{r\in\Z}$.
\begin{theorem}\label{theorem.ei61rw0}
The following identity holds for arbitrary $x\ne -1$ and integers $k$ and $n$:
\[
\begin{split}
\sum\limits_{j = 0}^n {x^j G_{j + k}^2 }  &= S_{F,n} (x)\left( {G_{k + 1}^2  + xG_k^2  + 2(1 - x)G_k G_{k + 1} } \right)\\
&\qquad+ \left( {1 - x^{n + 1} F_n^2 } \right)\left( {G_k^2  - 2G_k G_{k + 1} } \right)\\
&\qquad\quad + \left( {\frac{{1 + ( - 1)^n x^{n + 1} }}{{1 + x}}} \right)2G_k G_{k + 1}\,.
\end{split}
\]
\end{theorem}
\begin{proof}
Squaring the addition formula $G_{j+k}=F_{j-1}G_k+F_jG_{k+1}$ gives
\begin{equation}\label{eq.y54t300}
G_{j+k}^2=G_k^2F_{j-1}^2+G_{k+1}^2F_j^2+2G_kG_{k+1}F_jF_{j-1}\,.
\end{equation}
But,
\begin{equation}\label{eq.eswtba3}
\begin{split}
F_j F_{j - 1}  &= F_{j - 1} (F_{j + 1}  - F_{j - 1} ) = F_{j - 1} F_{j + 1}  - F_{j - 1}^2\\
&= F_j^2  - F_{j - 1}^2  + ( - 1)^j\,,
\end{split}
\end{equation}
by Cassini's identity.
Using \eqref{eq.eswtba3} in \eqref{eq.y54t300}, multiplying through by $x^j$ and summing over $j$ yields the identity of the theorem.
\end{proof}
An immediate application of Theorem \ref{theorem.ei61rw0} is to express $\sum_{j = 0}^n {x^j G_{j + k}^2 }$ in terms of $\sum_{j = 0}^n {x^j F_{j + k}^2 }$. Thus:
\begin{equation}
\begin{split}
\sum\limits_{j = 0}^n {x^j G_{j}^2 }  &= \left( {G_{1}^2  + xG_0^2  + 2(1 - x)G_0 G_{1} } \right)\sum\limits_{j = 0}^n {x^j F_{j}^2 }\\
&\qquad+ \left( {1 - x^{n + 1} F_n^2 } \right)\left( {G_0^2  - 2G_0 G_{1} } \right)\\
&\qquad\quad + \left( {\frac{{1 + ( - 1)^n x^{n + 1} }}{{1 + x}}} \right)2G_0 G_{1}\,.
\end{split}
\end{equation}
Setting $x=1$ in the identity of Theorem \ref{theorem.ei61rw0} produces
\begin{equation}\label{eq.yhmiqcx}
\begin{split}
\sum\limits_{j = 0}^n {G_{j + k}^2 }  &= F_n F_{n + 1} (G_{k + 1}^2  + G_k^2 ) + (F_n^2  - 1)(G_{k + 1}^2  - G_{k - 1}^2 )\\
&\qquad + (1 + ( - 1)^n )G_{k + 1} G_k\,;
\end{split}
\end{equation}
so that
\begin{equation}
\sum\limits_{j = 0}^{2n - 1} {G_{j + k}^2 }  = F_{2n - 1} F_{2n} (G_{k + 1}^2  + G_k^2 ) + (F_{2n - 1}^2  - 1)(G_{k + 1}^2  - G_{k - 1}^2 )
\end{equation}
and
\begin{equation}
\begin{split}
\sum\limits_{j = 0}^{2n} {G_{j + k}^2 }  &= F_{2n} F_{2n + 1} (G_{k + 1}^2  + G_k^2 ) + (F_{2n}^2  - 1)(G_{k + 1}^2  - G_{k - 1}^2 )\\
&\qquad + 2G_{k + 1} G_k\,.
\end{split}
\end{equation}
In particular,
\begin{equation}
\sum\limits_{j = 0}^n {F_{j + k}^2 }  = F_n F_{n + 1} F_{2k + 1}  + (F_n^2  - 1)F_{2k}  + (1 + ( - 1)^n )F_k F_{k + 1}\,;
\end{equation}
so that
\begin{equation}
\sum\limits_{j = 0}^{2n - 1} {F_{j + k}^2 }  = F_{2n} F_{2n - 1} F_{2k + 1}  + (F_{2n - 1}^2  - 1)F_{2k}
\end{equation}
and
\begin{equation}
\sum\limits_{j = 0}^{2n} {F_{j + k}^2 }  = F_{2n} F_{2n + 1} F_{2k + 1}  + (F_{2n}^2  - 1)F_{2k}  + 2F_k F_{k + 1}\,.
\end{equation}
\section{Sums of products}
It is convenient to introduce the notation
\begin{equation}\label{eq.be0cw7v}
A_n (x;k) = \sum\limits_{j = 0}^n {x^j G_{j + k}^2 }\,,
\end{equation}
with its evaluation as given in Theorem \ref{theorem.mofbl42}. Note that $S_{G,n}(x)=A_n (x;0)$.
\begin{theorem}\label{theorem.vjuhqd7}
The following identity holds for integers $n$, $s$, $k$ and arbitrary $x$:
\[
2F_s F_k \sum\limits_{j = 0}^n {x^j G_{j + k} G_{j + s} }  = F_s^2 A_n (x;k) + F_k^2 A_n (x;s) - F_{s - k}^2 S_{G,n} (x)\,.
\]

\end{theorem}
\begin{proof}
Multiply through \eqref{eq.cgd7f12} by $x^j$ and sum over $j$.
\end{proof}
In particular, setting $x=1$ in the identity of the theorem and making use of \eqref{eq.xei24hh} and \eqref{eq.pl8q4s9} produces
\begin{equation}\label{eq.mtiqlsj}
\begin{split}
&2F_s F_k \sum\limits_{j = 0}^n {G_{j + k} G_{j + s} }\\
&\qquad= \red{(} {F_s^2 (F_{k + 1} F_{k - 1}  + F_k^2 ) + F_k^2 (F_{s + 1} F_{s - 1}  + F_s^2 ) - F_{s - k}^2 } \red{)}\left( {G_n G_{n + 1}  - G_0 G_1  + G_0^2 } \right)\\
&\quad\qquad + F_s F_k (F_s F_{k - 1}  + F_k F_{s - 1} )(G_n  - G_1  + G_0 )(G_n  + G_1  - G_0 )\\
&\qquad\qquad + F_s F_k (F_s F_{k - 1}  + F_k F_{s - 1} )(G_{n + 1}  - G_0 )(G_{n + 1}  + G_0 )\,.
\end{split}
\end{equation}
Alternatively, setting $x=1$ in the identity of Theorem \ref{theorem.vjuhqd7} and making use of \eqref{eq.xei24hh} and \eqref{eq.yhmiqcx} gives
\begin{equation}
\begin{split}
2F_s F_k \sum\limits_{j = 0}^n {G_{j + k} G_{j + s} }  &= F_n F_{n + 1} \left( {F_s^2 (G_{k + 1}^2  + G_k^2 ) + F_k^2 (G_{s + 1}^2  + G_s^2 )} \right)\\
&\quad + \left( {F_n^2  - 1} \right)\left( {F_s^2 (G_{k + 1}^2  - G_{k - 1}^2 ) + F_k^2 (G_{s + 1}^2  - G_{s - 1}^2 )} \right)\\
&\quad + \left( {1 + ( - 1)^n } \right)\left( {F_s^2 G_{k + 1} G_k  + F_k^2 G_{s + 1} G_s } \right)\\
&\quad - F_{s - k}^2 \left( {G_n G_{n + 1}  - G_0 G_1  + G_0^2 } \right)\,;
\end{split}
\end{equation}
so that
\begin{equation}\label{eq.c0k4vfp}
\begin{split}
2F_s F_k \sum\limits_{j = 0}^{2n-1} {G_{j + k} G_{j + s} }  &= F_{2n-1} F_{2n} \left( {F_s^2 (G_{k + 1}^2  + G_k^2 ) + F_k^2 (G_{s + 1}^2  + G_s^2 )} \right)\\
&\quad + \left( {F_{2n-1}^2  - 1} \right)\left( {F_s^2 (G_{k + 1}^2  - G_{k - 1}^2 ) + F_k^2 (G_{s + 1}^2  - G_{s - 1}^2 )} \right)\\
&\quad - F_{s - k}^2 \left( {G_{2n-1} G_{2n}  - G_0 G_1  + G_0^2 } \right)
\end{split}
\end{equation}
and
\begin{equation}\label{eq.nvjuc4t}
\begin{split}
2F_s F_k \sum\limits_{j = 0}^{2n} {G_{j + k} G_{j + s} }  &= F_{2n} F_{{2n} + 1} \left( {F_s^2 (G_{k + 1}^2  + G_k^2 ) + F_k^2 (G_{s + 1}^2  + G_s^2 )} \right)\\
&\quad + \left( {F_{2n}^2  - 1} \right)\left( {F_s^2 (G_{k + 1}^2  - G_{k - 1}^2 ) + F_k^2 (G_{s + 1}^2  - G_{s - 1}^2 )} \right)\\
&\quad + 2\left( {F_s^2 G_{k + 1} G_k  + F_k^2 G_{s + 1} G_s } \right)\\
&\quad - F_{s - k}^2 \left( {G_{2n} G_{{2n} + 1}  - G_0 G_1  + G_0^2 } \right)\,.
\end{split}
\end{equation}
Identity \eqref{eq.mtiqlsj} and identities \eqref{eq.c0k4vfp} and \eqref{eq.nvjuc4t} subsume Berzsenyi's results \cite{berzsenyi}.
\begin{corollary}
The following identities hold for integer $n$ and arbitrary $x:$
\begin{equation}
\sum\limits_{j = 0}^n {x^j G_{j + 1} G_{j - 2} }  = (1 - x)S_{G,n} (x) + x^{n + 1} G_n^2  - (G_1  - G_0 )^2\,,
\end{equation}
\begin{equation}
2x\sum\limits_{j = 0}^n {x^j G_j G_{j - 1} }  = (1 - x - x^2 )S_{G,n} (x) + x^{n + 1} G_{n + 1}^2  + x^{n + 2} G_n^2  - x(G_1  - G_0 )^2  - G_0^2\,.
\end{equation}
\end{corollary}
In particular, we have
\begin{equation}
\sum\limits_{j = 0}^n {G_{j + 1} G_{j - 2} }  = (G_n  - G_1  + G_0 )(G_n  + G_1  - G_0 )
\end{equation}
and
\begin{equation}
2\sum\limits_{j = 0}^n {G_j G_{j - 1} }  = G_{n + 1} G_{n - 1}  + (G_n  - G_0 )(G_n  + G_0 ) + (G_1  - G_0 )(2G_0  - G_1 )\,.
\end{equation}
\begin{theorem}
The following identity holds for integers $k$ and $n$ and arbitrary $x$:
\[
\begin{split}
( - 1)^k 2x\sum\limits_{j = 0}^n {x^j G_{j + k} G_{j - k} }  &= \left( {xL_k^2  - (1 + x^2 )F_k^2  - 2xF_{k - 1} F_{k + 1} } \right)S_n (x)\\
&\qquad+ xF_k^2 \left( {x^{n + 1} G_n^2  - G_0^2  - G_1^2  + 2G_0 G_1 } \right)\\
&\qquad\quad- F_k^2 \left( {x^{n + 1} G_{n + 1}^2  - G_0^2 } \right)\,.
\end{split}
\]
\end{theorem}
In particular, we have
\begin{equation}\label{eq.fq6qy3p}
\begin{split}
( - 1)^k 2\sum\limits_{j = 0}^n {G_{j + k} G_{j - k} }  &= (L_k^2  - 2F_{k - 1} F_{k + 1}  - 2F_k^2 )(G_n G_{n + 1}  - G_0 G_1  + G_0^2 )\\
&\qquad - F_k^2 \left( {G_{n - 1} G_{n + 2}  - 2G_0 G_1  + G_1^2 } \right)\,.
\end{split}
\end{equation}
\begin{proof}
Setting $r=0$ in the identity (see Adegoke \cite[equation (2.12)]{adegoke18f})
\begin{equation}
F_{2k}G_{j+r}=F_{r+k}G_{j+k}-F_{r-k}G_{j-k}
\end{equation}
gives
\begin{equation}
L_kG_j=G_{j+k}+(-1)^kG_{j-k}\,,
\end{equation}
from which by squaring, we get
\begin{equation}\label{eq.rv1jt4e}
L^2_kG_j^2=G_{j+k}^2+G_{j-k}^2+2(-1)^kG_{j+k}G_{j-k}\,.
\end{equation}
Multiply \eqref{eq.rv1jt4e} by $x^{j+1}$, re-arrange and sum over $j$ to obtain
\[
( - 1)^k 2x\sum\limits_{j = 0}^n {x^j G_{j + k} G_{j - k} }  = xL_k^2 \sum\limits_{j = 0}^n {x^j G_j^2 }  - x\sum\limits_{j = 0}^n {x^j G_{j + k} }  - x\sum\limits_{j = 0}^n {x^j G_{j - k} }\,,
\]
from which the result follows by using the identity of Theorem \ref{theorem.mofbl42} to evaluate the last two sums on the right hand side.
\end{proof}
\begin{corollary}[Generating function of $G_{j + k} G_{j - k}$]
\[
\begin{split}
&( - 1)^k 2x\sum\limits_{j = 0}^\infty  {x^jG_{j + k} G_{j - k} }\\
&\qquad\qquad= \frac{{(xL_k^2  - (1 + x^2 )F_k^2  - 2xF_{k + 1} F_{k - 1} )(x^2 G_2^2  - (2x^2  + 2x - 1)G_0^2  - (2x^2  - x)G_1^2 )}}{{x^3  - 2x^2  - 2x + 1}}\\
&\qquad\qquad\qquad - xF_k^2 (G_1  - G_0 )^2  + F_k^2 G_0^2\,. 
\end{split}
\]
\end{corollary}

\section{Double binomial sums}\label{sec.binomial}
\begin{lemma}[{\cite[Lemma 5]{adegoke18c}}]\label{lem.h2de9i7}
Let $(X_r)$ be any arbitrary sequence, $X_r$ satisfying a four-term recurrence relation $hX_r=f_1X_{r-a}+f_2X_{r-b}+f_3X_{r-c}$, where $h$, $f_1$, $f_2$ and $f_3$ are arbitrary non-vanishing functions and $a$, $b$ and $c$ are integers. Then, the following identities hold:
\begin{equation}\label{eq.xfz9ytc}
\sum\limits_{j = 0}^n {\sum\limits_{i = 0}^j {\binom nj\binom jif_3^{n - j} f_2^{n + j - i} f_1^i X_{r - cn + (c - b)j + (b - a)i} } }  = h^nf_2^n X_r\,, 
\end{equation}
\begin{equation}
\sum\limits_{j = 0}^n {\sum\limits_{i = 0}^j {\binom nj\binom jif_2^{n - j} f_3^{n + j - i} f_1^i X_{r - bn + (b - c)j + (c - a)i} } }  = h^nf_3^n X_r\,,
\end{equation}
\begin{equation}\label{eq.vmi1nnt}
\sum\limits_{j = 0}^n {\sum\limits_{i = 0}^j {\binom nj\binom jif_1^{n - j} f_3^{n + j - i} f_2^i X_{r - an + (a - c)j + (c - b)i} } }  = h^nf_3^n X_r\,, 
\end{equation}
\begin{equation}
\sum\limits_{j = 0}^n {\sum\limits_{i = 0}^j {( - 1)^i \binom nj\binom jih^i f_3^{n - j} f_2^{j - i} X_{r - (c - a)n + (c - b)j + bi} } }  = ( - f_1 )^n X_r \,,
\end{equation}
\begin{equation}
\sum\limits_{j = 0}^n {\sum\limits_{i = 0}^j {( - 1)^i \binom nj\binom jih^i f_3^{n - j} f_1^{j - i} X_{r - (c - b)n + (c - a)j + ai} } }  = ( - f_2 )^n X_r
\end{equation}
and
\begin{equation}\label{eq.o7540wl}
\sum\limits_{j = 0}^n {\sum\limits_{i = 0}^j {( - 1)^i \binom nj\binom jih^i f_2^{n - j} f_1^{j - i} X_{r - (b - c)n + (b - a)j + ai} } }  = ( - f_3 )^n X_r \,.
\end{equation}

\end{lemma}
\begin{theorem}\label{thm.hm4tzsz}
The following identities hold for non-negative integer $n$ and integers $s$, $k$, $m$, $r$:
\begin{equation}
\begin{split}
&\sum\limits_{j = 0}^n {\sum\limits_{i = 0}^j {( - 1)^{i + (s + k + 1)j} \binom nj\binom jiF_s^{n - j + i} F_k^{n + j} F_m^{2n - i} F_{m - s}^{n - j} F_{m - k}^{n + j - i} F_{s - k}^i G_{r + kn + (s - k)j + (m - s)i}^2 } }\\
&\qquad= (F_m F_k F_{m - k}^2 F_{m - s} F_{s - k} )^n G_r^2\,,
\end{split}
\end{equation}
\begin{equation}
\begin{split}
&\sum\limits_{j = 0}^n {\sum\limits_{i = 0}^j {( - 1)^{j + (s + k)(i + j)} \binom nj\binom jiF_s^{n + j} F_k^{n - j + i} F_m^{2n - i} F_{m - s}^{n + j - i} F_{m - k}^{n - j} F_{s - k}^i G_{r + sn + (k - s)j + (m - k)i}^2 } }\\ 
&\qquad = (-1)^{(s+k-1)n}(F_m F_s F_{m - s}^2 F_{m - k} F_{s - k} )^n G_r^2\,, 
\end{split}
\end{equation}
\begin{equation}
\begin{split}
&\sum\limits_{j = 0}^n {\sum\limits_{i = 0}^j {( - 1)^{(s + k)(i + j) + i} \binom nj\binom jiF_s^{n + j - i} F_k^i F_m^{n + j} F_{m - s}^{n + j - i} F_{m - k}^i G_{r + mn + (k - m)j + (s - k)i}^2 } }\\
&\qquad = ( - 1)^{(s + k)n} (F_m F_s F_{m - k} F_{s - k} F_{m - s}^2 )^n G_r^2\,.
\end{split}
\end{equation}
\end{theorem}
\begin{proof}
Change $j$ to $r$ and re-arrange the identity of Theorem \ref{theorem.o1g1xkr} as
\[
\begin{split}
F_{m - s} F_{m - k} F_{s - k} G_r^2&=( - 1)^{s + k} F_k F_s F_{s - k} G_{r + m}^2\\
&\qquad+( - 1)^{s + k +1} F_k F_m F_{m - k} G_{r + s}^2\\
&\qquad\qquad+F_s F_m F_{m - s} G_{r + k}^2\,.
\end{split}
\]
In Lemma \ref{lem.h2de9i7} with $X_r=G_r^2$, set $h=F_{m - s} F_{m - k} F_{s - k}$, $f_1=( - 1)^{s + k} F_k F_s F_{s - k}$, $f_2=( - 1)^{s + k +1} F_k F_m F_{m - k}$, $f_3=F_s F_m F_{m - s} G_{r + k}^2$, $a=-m$, $b=-s$ and $c=-k$ in identities \eqref{eq.xfz9ytc} -- \eqref{eq.o7540wl}.
\end{proof}

\hrule

\noindent 2010 {\it Mathematics Subject Classification}:
Primary 11B37; Secondary 65B10, 11B65, 11B39.

\noindent \emph{Keywords: }
Fibonacci number, Lucas number, Fibonacci-like number, generating function, double binomial summation identity.

\hrule

\noindent Concerned with sequences:


\end{document}